\documentclass[11pt,reqno]{amsart}
\usepackage{amsmath, amsfonts, amssymb}
\usepackage[margin=3cm]{geometry}
\usepackage{mathtools}
\usepackage[dvipsnames, svgnames]{xcolor}
\usepackage[inline]{asymptote}
\usepackage{adjustbox, enumitem, tikz-cd, multirow}
\usepackage{mark-macros}
\usepackage{graphicx, subcaption}

\usepackage[pdfusetitle, pdfencoding=auto, psdextra, breaklinks=true]{hyperref}

\allowdisplaybreaks
\renewcommand{\textbf}[1]{{\bfseries\boldmath #1}}
\newcommand{\vocab}[1]{\textbf{\textcolor{BrickRed}{\boldmath #1}}}

\definecolor{mylinkcolor}{rgb}{0.0,0.0,0.7}
\definecolor{myurlcolor}{rgb}{0.0,0.0,0.7}
\hypersetup{
 colorlinks,
 urlcolor=myurlcolor,
 citecolor=myurlcolor,
 linkcolor=mylinkcolor,
 breaklinks=true
}

\usepackage{thmtools}
\usepackage{thm-restate}

\usepackage[nameinlink]{cleveref}

[name=Theorem]
[name=Question, sibling=theorem]
[name=Claim, sibling=theorem]
\declaretheorem{lemma}[name=Lemma, sibling=theorem]
\declaretheorem{proposition}[name=Proposition, sibling=theorem]
[name=Observation, sibling=theorem]
[name=Corollary, sibling=theorem]
\declaretheorem{definition}[style=definition, name=Definition, sibling=theorem]
[style=definition, name=Remark, sibling=theorem]
[style=definition, name=Example, sibling=theorem]
[style=definition, name=Open Problem, sibling=theorem]
\declaretheorem{conjecture}[style=definition, name=Conjecture, sibling=theorem]
\declaretheorem{proposition*}[name=Proposition, numbered=no]
\declaretheorem{theorem*}[name=Theorem, numbered=no]

\usepackage[backend=biber, style=alphabetic, sorting=nyt,
maxnames=20, maxalphanames=20, backref, url=true, doi=true]{biblatex}
\DeclareSourcemap{
  \maps[datatype=bibtex]{
    \map[overwrite]{
      \step[fieldsource=doi, final]
      \step[fieldset=url, null]
      \step[fieldset=eprint, null]
    }  
  }
}
\DeclareSourcemap{
  \maps[datatype=bibtex]{
    \map{
      \step[fieldset=issn, null]
    }
  }
}
\DefineBibliographyStrings{english}{%
  backrefpage = {$\uparrow$ p.},%
  backrefpages = {$\uparrow$ pp.}%
}
\defbibheading{bibliography}[\refname]{\section*{#1}}

\title{Sparse $k$-AP Covering Sets and the Arithmetic Kakeya Conjecture}

\author{Pitchayut Saengrungkongka}
\address{
  Department of Mathematics,
  Harvard University,
  Cambridge,
  MA 02138,
  USA
}
\email{saengrun@math.harvard.edu}

\begin{document}
\begin{abstract}
A subset $A\subseteq \NN_0$ is $k$-AP covering
if there exists a constant $n_0$ such that 
for every integer $x>n_0$, there exists $d\in\NN_0$ such that 
 $x-d, x-2d,\dots,x-(k-1)d$ are all in $A$.
Disproving a conjecture of Kiss, S\'andor, and Yang,
we prove that for every integer $k\geq 6$, there exists a constant 
$\eps=\eps_k>0$ and a $k$-AP covering set 
$A$ such that 
$|A\cap \{0,1,\dots,n\}| <  n^{\frac{k-2}{k-1}-\eps}$
for all sufficiently large $n$.
We also relate this problem to the Arithmetic 
Kakeya Conjecture by Katz and Tao.
\end{abstract}
\maketitle

\section{Introduction}

\begin{definition}
A subset $A\subseteq \NN_0$ is \vocab{$k$-AP covering}
if there exists $n_0$ such that for all $x>n_0$,
there exists $d\in\NN_0$ such that 
 $x-d, x-2d,\dots,x-(k-1)d$ are all in $A$.
\end{definition}
We are interested in studying how sparse a $k$-AP covering set can be. 
This study was initiated by 
Kiss, S\'andor, and Yang \cite{generalized} 
in an effort to obtain a bound for generalized Stanley sequences.
More specifically, suppose that $A_0=\{a_1 < a_2 < \dots < a_t\}$
is a finite set of nonnegative integers containing no nonconstant arithmetic 
progression of length $k$, and for each $i>t$,
we let $a_i$ be the smallest integer $m>a_{i-1}$ such that 
$\{a_1,\dots,a_{i-1},m\}$ does not contain any nonconstant 
arithmetic progression of length $k$.
The sequence $a_1,a_2,\dots$ is called the \vocab{generalized 
Stanley sequence} $S_k(A_0)$, and it is easy to see that 
the set of terms in a generalized Stanley sequence
$\{a_1,a_2,\dots\}$ is $k$-AP covering.
Thus, a bound on $k$-AP covering sets would imply 
a bound on generalized Stanley sequences.

In the case of $k=3$, an easy counting argument proves 
the following proposition. 
\begin{proposition}[{\cite[p.~2]{generalized}}]
For any $3$-AP covering set $A\subseteq\NN_0$
and $\eps>0$, we have
$$|A\cap \{0,1,\dots,n\}| > (\sqrt 2-\eps)\sqrt n$$
for all sufficiently large integers $n$.
\end{proposition}
In particular, this proposition implies that the 
$n$-th term of the sequence $S_3(A_0)$ is at most 
$\left(\tfrac 12+o(1)\right)n^2$.
This counting argument, in the specific setting of 
$S_3(A_0)$, appears in papers of Gerver and Ramsey 
\cite{quadratic_growth} and Moy \cite{quadratic_growth_2}.
This remains the best-known bound on the $n$-th element 
of $S_3(A_0)$ to date.
Determining the best possible upper bound 
on the $n$-th term of $S_3(A_0)$ is 
one of Erd\H os's open problems \cite{erdos_271}.

Kiss, S\'andor, and Yang proved that this bound for $3$-AP 
covering sets is tight up to a constant factor, 
and their construction was subsequently improved by 
Chen and Fang \cite{ap3_covering_1, ap3_covering_2}.
Kiss, S\'andor, and Yang also formulated the following conjecture 
for general $k$.
\begin{conjecture}[{\cite[Conj.~1]{generalized}}]
\label{conj:general_k}
There exists a constant $c_k>0$ such that 
for any $k$-AP covering set $A\subseteq\NN_0$, we have
$$|A\cap \{0,1,\dots,n\}| > c_k n^{\frac{k-2}{k-1}}$$
for infinitely many nonnegative integers $n$.
\end{conjecture}
Note that the exponent $\frac{k-2}{k-1}$ cannot be 
replaced by any larger exponent, as shown in 
\cite[Thm.~3]{generalized} by picking a random set $A$,
including a positive integer $n$ with probability 
$\operatorname{\Theta}\left((\log n\,/\,n)^{1/(k-1)}\right)$.
Thus, \Cref{conj:general_k} is saying that 
this random construction yields an optimal exponent.
If \Cref{conj:general_k} is true, we would get that 
the $n$-th term of the generalized Stanley sequence 
$S_k(A_0)$ is at most $O\left(n^{\frac{k-1}{k-2}}\right)$.

Our result (proven in \Cref{sec:fourier})
disproves \Cref{conj:general_k} for every integer $k\geq 6$, 
leaving only the case $k\in\{4,5\}$ open.

\begin{restatable}{theorem}{fourier}
\label{thm:fourier}
For every integer $k\geq 6$, 
there exists a constant $\eps=\eps_k>0$ 
and a $k$-AP covering set 
$A$ such that 
$$|A\cap \{0,1,\dots,n\}| <  n^{\frac{k-2}{k-1}-\eps}$$
for all sufficiently large positive integers $n$.
\end{restatable}
The proof uses a construction based on Fourier analysis 
on $\ZZ/N\ZZ$, where $N$ is a product of several prime numbers.
This construction was partially inspired by the construction of
sets in which a linear pattern has no popular common difference
due to Sah, Sawhney, and Zhao in \cite{popular}.
In our case, we want to construct a set having no 
\emph{unpopular} common difference instead.

We also give a different construction 
in \Cref{sec:arithmetic} that yields a substantially 
stronger exponent for all sufficiently large $k$.

\begin{restatable}{theorem}{arithmetic}
\label{thm:arithmetic}
There exists a constant $c>0$ such that for all
sufficiently large $k$, there exists a $k$-AP covering set 
$A$ such that 
$$|A\cap \{0,1,\dots,n\}| <  n^{1-\frac{c}{\log\log k}}$$
for all sufficiently large positive integers $n$.
\end{restatable}
The proof of this result (given in \Cref{sec:arithmetic}) 
is a close modification of a construction 
due to Green and Ruzsa \cite[Thm.~1.2]{arithmetic_kakeya}
for a related problem arising from the Arithmetic Kakeya 
Conjecture. 
We discuss the connection with the Arithmetic Kakeya 
Conjecture in the following subsection.

\subsection{Relation to Arithmetic Kakeya Conjecture}
\label{subsec:relation}
The problem of determining how sparse
a $k$-AP covering set can be is closely related to 
the Arithmetic Kakeya Conjecture,
which was proposed by Katz and Tao 
\cite{katz_tao} in an effort to resolve Kakeya's conjecture 
using an approach from additive combinatorics.

Katz and Tao's original formulation of the conjecture is as follows.
For any subset $B\subseteq \ZZ^2$ and $r\in\QQ\cup\{\infty\}$, define 
$$\pi_r(B) = \begin{cases}
\{x+ry : (x,y)\in B\} &  \text{ if } r\neq\infty \\
\{y : (x,y)\in B\} & \text{ if } r=\infty.
\end{cases}$$
We then have the following conjecture.
\begin{conjecture}[Arithmetic Kakeya Conjecture]
\label{conj:arithmetic_kakeya}
For any $\eps>0$, there exists a positive integer $\ell$ and 
$r_1,\dots,r_\ell\in\QQ\cup\{\infty\}$, none equal to $-1$, 
such that 
$$|\pi_{-1}(B)| \leq \max_{i=1}^\ell |\pi_{r_i}(B)|^{1+\eps}$$
for all finite subsets $B\subseteq \ZZ^2$.
\end{conjecture}
See \cite{needles} for an explanation of how this conjecture 
is related to the original Kakeya Conjecture.

Green and Ruzsa \cite[Thm.~1.1]{arithmetic_kakeya} give several 
equivalent formulations of the Arithmetic Kakeya Conjecture.
The one given above is \cite[Conj.~3]{arithmetic_kakeya}.
Another formulation worth mentioning is \cite[Conj.~1]{arithmetic_kakeya}, which 
states that for any $\eps>0$, there exists $k_0$ such that 
for any $k > k_0$, there exists 
a constant $c_k>0$ such that if $S$ is a set that contains,
for every $d\in\{1,2,\dots,n\}$, 
an arithmetic progression of length $k$ and common 
difference $d$,
then we have $|S| > c_kn^{1-\eps}$.

By adapting the argument from \cite[\S 2]{arithmetic_kakeya}
(specifically, in the proof that Conjecture 2 implies 
Conjecture 1'), we show (in \Cref{sec:kakeya}) 
that the Arithmetic Kakeya Conjecture implies a bound on AP-covering sets.
\begin{restatable}{theorem}{kakeya}
\label{thm:arithmetic_kakeya}
Assume that \Cref{conj:arithmetic_kakeya} is true.
Then for any $\eps>0$, there exists $k_0$ such that 
for any $k>k_0$, if $A$ is a $k$-AP covering set,
there exists $n_0$ (possibly depending on $A$) such that
$$|A\cap \{0,1,\dots,n\}| > n^ {1-\eps}$$
for all $n>n_0$.
\end{restatable}

We expect that the converse of \Cref{thm:arithmetic_kakeya} 
can be obtained by following the remainder of the argument in 
\cite[\S 2]{arithmetic_kakeya}. We do not need this converse here.
This result suggests that the problem of determining the 
optimal exponent of the sparsest $k$-AP covering sets is difficult.

\subsection{Notation}

We let $\NN$ be the set of positive integers, 
$\NN_0$ be the set of nonnegative integers,
and $\mathbb Z/N\mathbb Z$ denote the ring of integers modulo $N$.
A \vocab{$k$-AP} is an arithmetic progression of length $k$, 
which may be constant (i.e., all terms equal)
unless otherwise stated.

For any function $f : \mathbb Z/N\mathbb Z\to \mathbb R$,
we let $\mathbb E_x[f(x)]$ denote the expected value 
of $f(x)$, where $x$ is chosen uniformly at 
random from $\mathbb Z/N\mathbb Z$.

We use the standard asymptotic notation
($O$, $\Omega$, and $\Theta$) throughout this paper.

\section{Fourier-Analytic Construction: 
Proof of \texorpdfstring{\Cref{thm:fourier}}{Theorem \ref{thm:fourier}}}
\label{sec:fourier}
We prove \Cref{thm:fourier}.
Throughout this section, we fix an integer $k\geq 6$.
Before delving into the details, let us briefly outline 
the whole proof.
The proof of \Cref{thm:fourier} proceeds in five steps.
\begin{enumerate}
\item \label{item:fourier} 
First, we construct a function $f_p : \ZZ/p\ZZ\to (0,\infty)$
such that $\mathbb E_x[f_p(x)]=1$
and $\mathbb E_x[f_p(a-3x)f_p(a-4x)f_p(a-5x)] > 1.001$
for all $a\in\ZZ/p\ZZ$ using a Fourier-analytic construction.
\item \label{item:crt}
Using the Chinese remainder theorem, we construct 
a function $f : \ZZ/Q\ZZ \to (0,\infty)$
(where $Q$ is a product of several primes)
such that $\mathbb E_x[f(x)]=1$
but 
$$\mathbb E_x[f(a-3x)f(a-4x)f(a-5x)] > 100^{k^2} \log Q,$$
giving a $\log Q$-factor and a constant factor gain.
The $\log Q$ factor is necessary 
for the step (\ref{item:sampling}) to work.
\item \label{item:smoothing}
By modifying $f$, we construct a function 
$g : \ZZ/Q\ZZ \to (0,\infty)$ such that 
$\mathbb E_x[g(x)]=1$ and 
$$\mathbb E_x\left[\prod_{j=1}^{k-1} g(a-jx)\right] 
> 50^{k^2}\log Q.$$
It is worth noting that we are sacrificing
some constant factor gain from step (2)
in order to control other elements 
$a-jx$ for $j\notin \{3,4,5\}$.
This does not affect our construction since the 
gain in step (\ref{item:crt}) can be made arbitrarily large.
\item \label{item:sampling}
By carefully sampling according to the distribution determined by $g$,
we construct a subset $B\subseteq \{0,1,\dots,N-1\}$ such that 
$|B| < \frac 12 N^{\frac{k-2}{k-1}}$
and for any $a\in \{0,1,\dots,N-1\}$, there exists $x$ such that 
$a-x$, $a-2x$, \dots, $a-(k-1)x$ are all in $B$.
\item \label{item:tensor}
Finally, by the tensor power trick, we 
construct the desired $k$-AP covering set $A$.
\end{enumerate}

Below, each step will be presented as a lemma,
and the formal statements can be found there.
The following lemma handles step (\ref{item:fourier}).
\begin{lemma}
\label{lem:fourier}
For every prime $p>100$, there exists a function $f_p  :\ZZ/p\ZZ\to \mathbb R$ such that
\begin{enumerate}[label=(\alph*)]
    \item $0<f_p(x) < 2$ for all $x\in\ZZ/p\ZZ$.
    \item $\mathbb E_x[f_p(x)] = 1$.
    \item For any $a\in \mathbb Z/p\mathbb Z$, we have 
    $$\mathbb E_x[f_p(a-3x)f_p(a-4x)f_p(a-5x)] 
    > 1.001.$$
\end{enumerate}
\end{lemma}
\begin{proof}
Let $\omega = \exp(2\pi i/p)$. 
We take the function $f_p$ of the form
$$f_p(x) = 1+0.2(\omega^x+\omega^{-x}) + 0.1(\omega^{2x}+\omega^{-2x}),$$
so that $0.4 \leq f_p(x)\leq 1.6$ for all $x$,
which means that (a) is satisfied.
Since $\mathbb E_x[\omega^{bx}]=0$ for all $b$ not divisible by $p$,
we get that $\mathbb E_x[f_p(x)]=1$, verifying (b).

Now, we have to verify (c).
To do this, we note that after fully expanding 
$f_p(a-3x)f_p(a-4x)f_p(a-5x)$, we get a linear combination of
$\omega^{\eps_1(a-3x) + \eps_2(a-4x) + \eps_3(a-5x)}$
for some $\eps_1,\eps_2,\eps_3\in\{0,1,-1,2,-2\}$.
If $3\eps_1 + 4\eps_2 + 5\eps_3 \neq 0$,
then after taking expected value, 
we get a multiple of $\mathbb E_x[\omega^{bx}]$
for some integer $b$ with $|b|<100$, 
so such terms vanish.
It can be directly checked that
the only combinations $(\eps_1,\eps_2,\eps_3)$
that make $3\eps_1+4\eps_2+5\eps_3=0$ are 
$$(\eps_1,\eps_2,\eps_3) \in \{(0,0,0), 
(1,-2,1), (-1,2,-1), (2,1,-2), (-2,-1,2)\}.$$
The combinations $(1,-2,1)$ and $(-1,2,-1)$
each contribute $(0.2)^2 (0.1)$ to the sum.
The combinations 
$(2,1,-2)$ and $(-2,-1,2)$ contribute
$(0.2)(0.1)^2 \omega^a$ and $(0.2)(0.1)^2 \omega^{-a}$
to the sum, respectively.
Therefore, we get that 
$$\mathbb E_x[f_p(a-3x)f_p(a-4x)f_p(a-5x)]
= 1.008 + 0.002(\omega^a + \omega^{-a}) 
> 1.001,$$
as desired.
\end{proof}

Next, we handle step (\ref{item:crt})
by using the Chinese remainder theorem to stitch together 
$f_p$ for several primes $p$.
\begin{lemma}
\label{lem:crt}
There are infinitely many positive integers $Q$ 
for each of which there exists 
a function $f : \ZZ/Q\ZZ \to (0,\infty)$ such that 
\begin{enumerate}[label=(\alph*)]
\item $0<f(x) < Q^{\frac 1{100k}}$ for all $x\in \ZZ/Q\ZZ$.
\item $\mathbb E_x[f(x)] = 1$.
\item for any $a\in \ZZ/Q\ZZ$, we have
$$\mathbb E_x[f(a-3x)f(a-4x)f(a-5x)] > 100^{k^2}\log Q.$$
\end{enumerate}
\end{lemma}
\begin{proof}
Let $p_1,p_2,\dots,p_n$ be the first $n$ primes 
greater than $2^{100k}$,
where $n$ is sufficiently large (to be specified later).
Take $Q=p_1\cdots p_n$ and 
$$f(x) = f_{p_1}(x) \cdots f_{p_n}(x),$$
where $f_{p_i}$ are as in \Cref{lem:fourier}.
Since $f_{p_i}(x) < 2 < p_i^{\frac 1{100k}}$,
we get that $f(x) < Q^{\frac 1{100k}}$, verifying (a).

To verify (b), we note by the Chinese remainder theorem 
that if $x$ is uniformly chosen in $\ZZ/Q\ZZ$,
then the random variables $x\bmod p_1, \dots, x\bmod p_n$ 
are independent and uniformly distributed. This implies that
$$\mathbb E_x[f(x)] = \mathbb E_x[f_{p_1}(x)] 
\cdots \mathbb E_x[f_{p_n}(x)] = 1,$$
verifying (b).

Finally, to verify (c), we note by the same 
Chinese remainder theorem argument that 
\begin{align*}
    \mathbb E_x[f(a-3x)f(a-4x)f(a-5x)]
    &= \prod_{i=1}^n \mathbb E_x[f_{p_i}(a-3x)f_{p_i}(a-4x)
    f_{p_i}(a-5x)] \\
    &> 1.001^n.
\end{align*}
By a variant of the prime number theorem,
the product of the first $n$ primes greater than $2^{100k}$
is $\exp(O(n\log n))$.
Thus, we get that $\log Q = O(n\log n)$,
which is less than $100^{-k^2} 1.001^n$ for all sufficiently large $n$.
Thus, $1.001^n > 100^{k^2}\log Q$ if $n$ is sufficiently large,
implying existence of such $Q$.
\end{proof}
We now handle step (\ref{item:smoothing}) by smoothing the function $f$,
yielding a function $g$ that controls $a-jx$ 
for all $j\in\{1,2,\dots,k-1\}$
(at the cost of losing a constant factor from the bound,
but the loss is offset by the amplification in 
step (\ref{item:crt})).
The reason we need $k\geq 6$ is because the terms $a-3x$, $a-4x$,
and $a-5x$ inside $\{a-x,a-2x,\dots,a-(k-1)x\}$ are necessary 
in order to use the construction in step (\ref{item:fourier}).
One cannot find a suitable construction in 
\Cref{lem:fourier} if we only have 
the terms $a-x$, $a-2x$, $a-3x$, and $a-4x$.
\begin{lemma}
\label{lem:smoothing}
For each $Q$ in \Cref{lem:crt},
there exists a function $g : \ZZ/Q\ZZ\to (0,\infty)$
such that 
\begin{enumerate}[label=(\alph*)]
\item $0<g(x) < Q^{\frac 1{100k}}$ for all $x\in \ZZ/Q\ZZ$.
\item $\mathbb E_x[g(x)] = 1$.
\item for any $a\in \ZZ/Q\ZZ$, we have
$$\mathbb E_x\left[\prod_{j=1}^{k-1} g(a-jx)\right] > 50^{k^2}\log Q.$$
\end{enumerate}
\end{lemma}
\begin{proof}
We take $g(x) = \frac{1+f(x)}2$,
where $f$ is as in \Cref{lem:crt},
which clearly satisfies (a) and (b).
To verify (c), we use the bounds $g(a-jx) \geq \frac{f(a-jx)}2$
for $j\in\{3,4,5\}$ and $g(a-jx) \geq \frac 12$
for $j\notin\{3,4,5\}$, giving
\begin{align*}
\mathbb E_x\left[\prod_{j=1}^{k-1} g(a-jx)\right]
&\geq \mathbb E_x\left[\frac 1{2^{k-1}} f(a-3x)f(a-4x)f(a-5x)\right] \\
&\geq \frac 1{2^{k-1}}100^{k^2}\log Q 
> 50^{k^2}\log Q, \qedhere
\end{align*}
since $k\geq 6$.
\end{proof}

We now sample from $g$ to get a subset of integers.
\begin{lemma}
\label{lem:sampling}
There exist a positive integer $N$ 
and a set $B\subseteq \{0,1,\dots,N-1\}$
such that 
\begin{enumerate}[label=(\alph*)]
\item $|B| < \frac 12 N^{\frac{k-2}{k-1}}$ and 
\item for any $a\in \{0,1,\dots,N-1\}$, there exists 
a nonnegative integer $d$ such that
$a-d$, $a-2d$, \dots, $a-(k-1)d$ are all in $B$.
\end{enumerate}
\end{lemma}
\begin{proof}
We take $Q$ and $g$ as in \Cref{lem:smoothing}.
Let $N=10k^2Q^2$. 
We construct the set $B$ as follows:
\begin{itemize}
    \item include each $x\in [0,10k^2Q)$;
    \item include each $x\in [10k^2Q, 20k^2Q)$ 
    with probability $\frac 1{100k^2} Q^{-\frac 1{k-1}}g(x)$;
    \item include each $x\in [20k^2Q, 30k^2Q)$
    with probability $\frac 1{100k^2} (2Q)^{-\frac 1{k-1}} g(x)$;
    \item include each $x\in [30k^2Q, 40k^2Q)$
    with probability $\frac 1{100k^2} (3Q)^{-\frac 1{k-1}} g(x)$;
    \item \hspace{2cm} \vdots 
    \item include each $x\in [10(Q-1)k^2Q, 10k^2Q^2)$
    with probability $\frac 1{100k^2} ((Q-1)Q)^{-\frac 1{k-1}} g(x)$.
\end{itemize}
In particular, we partition $\{0,1,\dots,N-1\}$ 
into $Q$ blocks, each of which has length $10k^2Q$.
Note that 
\begin{align*}
\mathbb E[|B|] &= 10k^2Q + 
\Big(1+2^{-\frac 1{k-1}}+3^{-\frac 1{k-1}} + \dots + (Q-1)^{-\frac 1{k-1}}\Big)
\cdot \left(\frac 1{100k^2}Q^{-\frac 1{k-1}}\right) \sum_{x=0}^{10k^2Q-1} g(x)\\
&< 20k^2Q  + \left(\frac{k-1}{k-2}\cdot  Q^{\frac{k-2}{k-1}}\right)
\cdot \left(\frac 1{100k^2}  Q^{-\frac 1{k-1}}\right) \cdot 
\left(10k^2Q \cdot \mathbb E_x[g(x)]\right) \\
&< 20k^2 Q + \frac 15 Q^{2\cdot \frac{k-2}{k-1}} \\
&< \frac 14 Q^{2\cdot \frac{k-2}{k-1}},
\end{align*}
if $Q$ is sufficiently large.
Thus, by Markov's inequality, we deduce that 
$$\mathbb P\left(|B| > \tfrac 12 N^{\frac{k-2}{k-1}}\right) < 0.1.$$

Next, we look at the second condition.
Let $a\in\{0,1,\dots,N-1\}$, and take $n$ such that
$a\in [10k^2nQ, 10k^2(n+1)Q)$.
If $n=0$, then automatically we have $a\in B$,
so we may pick $d=0$.

Henceforth, we assume that $n\geq 1$.
We commit to picking $x$ from the interval $[knQ, (k+1)nQ]$. 
We claim that for any distinct $x_1, x_2\in [knQ, (k+1)nQ]$,
the sets $\{a-x_1, a-2x_1,\dots,a-(k-1)x_1\}$
and $\{a-x_2,a-2x_2,\dots,a-(k-1)x_2\}$ are disjoint.
Suppose that they are not disjoint. Then there exists 
$u,v\in\{1,2,\dots,k-1\}$ such that 
$ux_1 = vx_2$, so $\tfrac{x_1}{x_2} = \tfrac uv$.
However, $\tfrac{x_1}{x_2}$ is between $\tfrac k{k+1}$
and $\tfrac{k+1}k$, which cannot be equal to 
$\tfrac uv$ for $u,v\in\{1,2,\dots,k-1\}$
(unless $u=v$).
Therefore, $u=v$. From $ux_1=vx_2$, we then conclude that 
$x_1=x_2$, which is a contradiction.

Thus, we deduce that the events of the form 
$\{a-x, a-2x,\dots,a-(k-1)x\} \subseteq B$ 
are independent across all $x$.
Furthermore, since $a-jx < 10k^2(n+1)Q$
(so it is in at most the $(n+1)$-th block)
and the prefactors assigned to each block is decreasing,
we get that the probability that $a-jx\in B$ is at least 
$\frac 1{100k^2}(nQ)^{-\frac{1}{k-1}}g(a-jx)$. 

Therefore, the probability that $a$ does not 
satisfy the second condition is
\begin{align*}
\mathbb P(a\text{ fails}) 
&= \prod_{x=knQ}^{(k+1)nQ} \Big(1 - \prod_{j=1}^{k-1} 
\mathbb P[a-jx\in B]\Big) \\
&\leq \prod_{x=knQ}^{(k+1)nQ} \Big(1 - \prod_{j=1}^{k-1} 
\left(\frac 1{100k^2}\cdot 
(nQ)^{-\frac 1{k-1}} g(a-jx)\right)\Big) \\
&= \prod_{x=knQ}^{(k+1)nQ} \Big(1 - 
\frac 1{(100k^2)^{k-1} nQ} \prod_{j=1}^{k-1} g(a-jx)\Big) \\
&< \prod_{x=knQ}^{(k+1)nQ} \Big(1 - 
\frac 1{10^{k^2} nQ} \prod_{j=1}^{k-1} g(a-jx)\Big) \\
&\leq \prod_{x=knQ}^{(k+1)nQ}  \exp\left(-\frac 1{10^{k^2}nQ}
\prod_{j=1}^{k-1} g(a-jx) \right) \tag{using $\exp(-x)\geq 1-x$}\\
&\leq \exp\left(-\frac 1{10^{k^2}nQ} \sum_{x=knQ}^{(k+1)nQ}
\prod_{j=1}^{k-1} g(a-jx) \right)\\
&\leq \exp\Big(-\frac 1{10^{k^2}} 
\mathbb E_x \left[\prod_{j=1}^{k-1} g(a-jx)\right]\Big) \\
&\leq \exp\Big(-100\log Q\Big) < \frac 1{Q^{100}},
\end{align*}
so summing across all $a\in \{0,1,\dots,N-1\}$ 
and using the union bound gives 
that the probability that the second condition fails is at most $\frac 12$.
Thus, by the union bound, there is a nonzero probability
that $B$ satisfies both conditions.
Therefore, such a $B$ exists.
\end{proof}

Finally, we prove \Cref{thm:fourier},
which we reproduce the statement below.
\fourier*
\begin{proof}[Proof of \Cref{thm:fourier}]
Take $N$ and $B$ as in \Cref{lem:sampling}.
We let 
$$A = \left\{\sum_{i=0}^m e_i N^i 
: m\geq 0 \text{ and }e_0,\dots,e_m\in B\right\}$$
be the set of nonnegative integers whose base-$N$
representations have all digits in $B$.

We first show that $A$ is $k$-AP covering.
Indeed, given $x=\sum_{i=0}^m x_iN^i$ where 
$x_i\in \{0,1,\dots,N-1\}$,
for each $i$, we take $d_i$ such that $x_i-jd_i\in B$
for all $j\in\{1,2,\dots,k-1\}$.
Then we may take $d=\sum_{i=0}^m d_iN^i$ to get that 
$x-jd = \sum_{i=0}^m (x_i-jd_i)N^i\in A$
for all $j\in\{1,2,\dots,k-1\}$.

We now bound the size of $A$.
Note that 
$$|A\cap \{0,1,\dots,n\}| 
\leq |B|^{\lfloor\log_N n\rfloor+1} 
\leq |B| \cdot n^{\frac{\log|B|}{\log N}},$$
and since $\frac{\log |B|}{\log N} 
< \frac{k-2}{k-1}$, it would suffice to pick 
any $\eps>0$ such that $\frac{\log |B|}{\log N} 
< \frac{k-2}{k-1}-\frac{\eps}2$
and take $n$ sufficiently large so that $|B| < n^{\eps/2}$.
\end{proof}
\section{Arithmetic Construction: Proof of \texorpdfstring{\Cref{thm:arithmetic}}{Theorem \ref{thm:arithmetic}}}
\label{sec:arithmetic}
In this section, we prove \Cref{thm:arithmetic}.
The proof is very similar to 
Green and Ruzsa's construction in \cite[Thm.~1.2]{arithmetic_kakeya}.
In particular, we start by constructing a subset 
that satisfies the condition for $\{0,1,\dots,N-1\}$.
This set is a union of sets of the form 
$\{x-d(x), x-2d(x),\dots,x-(k-1)d(x)\}$,
where $d(x)=x^2\bmod Q$ and $Q$ is a product of several primes.
This is so that each $x-j\,d(x)$ is confined 
in a small number of residue classes modulo $Q$,
which makes $B$ small enough to meet the required bound.

Finally, we use the base-$N$ representation 
construction (similar to the last step of 
the proof of \Cref{thm:fourier}) 
to get the desired AP-covering set.

The following lemma constructs such a subset of 
$\{0,1,\dots,N-1\}$.
\begin{lemma}
\label{lem:arithmetic}
There exists a constant $c>0$ 
such that for all sufficiently large $k$,
there exists a positive integer $Q$, $N$ and a subset 
$B\subseteq \{0,1,\dots,N-1\}$ such that 
$$|B| < N^{1-\frac{c}{\log\log k}}$$
and for every $x\in\{0,1,\dots,N-1\}$,
there exists a nonnegative integer $d$ such that 
$x-d$, $x-2d$, \dots, $x-(k-1)d$ are all in $B$.
\end{lemma}
\begin{proof}
Let $m=\lceil 100\log k\rceil$, $3=p_1<p_2<\dots$
be the odd primes, and $Q=p_1p_2\cdots p_m$.
We take $N=Q^2$ and
$$B = \{0,1,\dots,kQ-1\}\cup \bigcup_{x=kQ}^{Q^2-1} \{x-d(x), x-2d(x),\dots, 
x-(k-1)d(x)\}$$
where $d(x)=x^2\bmod Q$ is the unique integer $0\leq r < Q$
such that $x^2\equiv r\pmod Q$.

Clearly, this satisfies the second condition,
so it suffices to bound the size of the set.
To bound the size of the union, 
we note that for all $j$ and prime $p_i$
for $1\leq i\leq m$ such that $p_i\nmid j$,
we have that 
$$x-j\,d(x) \equiv x-jx^2 \equiv \frac{1}{4j} 
- j\left(x-\frac 1{2j}\right)^2\pmod{p_i}.$$
Thus, if $p_i\nmid j$, then the values of $x-j\,d(x)$
lie in an affine image of the quadratic residues modulo $p_i$,
so $x-j\,d(x)$ takes at most $\frac{p_i+1}2$ 
distinct values modulo $p_i$.
Therefore, modulo $Q$, $x-j\,d(x)$ takes at most 
\begin{align*}
\prod_{\substack{i\in\{1,2,\dots,m\} \\ 
p_i\nmid j}}\frac{p_i+1}2 
\cdot \prod_{\substack{i\in\{1,2,\dots,m\} \\ 
p_i\mid j}} p_i 
\ \leq\ \frac{Q}{2^{m-\omega(j)}} 
\prod_{i=1}^m \left(1+\frac 1{p_i}\right)
\ \leq\ \frac Q{2^{m/2}} \cdot k
\ \leq\ \frac Q{k^4}
\end{align*}
distinct values.
Here, we use the bounds $\omega(j) = O\left(\frac{\log k}{\log\log k}\right) \ll m$
and $\prod_{i=1}^m \left(1+\frac 1{p_i}\right) 
= O(\log m) \ll k$.

Thus, we conclude that $x-j\,d(x)$ takes
at most $\frac Q{k^4}$ distinct residues modulo $Q$.
Taking the union over all $j\in\{1,2,\dots,k-1\}$
gives that the union in the definition of $B$
takes at most $\frac Q{k^3}$ distinct residues modulo $Q$.
Since $0\leq x-jd(x) < Q^2$ for all $kQ\leq x<Q^2$ 
and $1\leq j\leq k-1$, we get that every element of $B$ is 
in $[0, Q^2-1]$.
Each residue modulo $Q$ has exactly $Q$ numbers 
in the interval $[0,Q^2-1]$,
so there are at most $Q\cdot \frac Q{k^3}=\frac{Q^2}{k^3}$
possible elements in the union defining $B$. Thus,
$$|B| \leq kQ + \frac{Q^2}{k^3}.$$
Next, we note that by the prime number theorem,
$Q = \exp(\Theta(m\log m))$.
Thus, $Q = k^{\Theta(\log\log k)}$.
In particular, $Q>k^4$ for all sufficiently
large $k$, so 
$$|B| \leq \frac{2Q^2}{k^3} < 
Q^{2-\frac{2c}{\log\log k}} = N^{1-\frac{c}{\log\log k}}$$
for some constant $c>0$,
implying that $B$ satisfies the first condition.
\end{proof}

We now prove \Cref{thm:arithmetic},
which we reproduce the statement below.
\arithmetic*
\begin{proof}[Proof of \Cref{thm:arithmetic}]
Take $N$ and $B$ as in \Cref{lem:arithmetic}.
We let 
$$A = \left\{\sum_{i=0}^m e_i N^i 
: m\geq 0 \text{ and }e_0,\dots,e_m\in B\right\}$$
be the set of nonnegative integers whose base-$N$
representations have all digits in $B$.
By a similar argument as in the proof of \Cref{thm:fourier},
we get that $A$ is $k$-AP covering.
Note that 
$$|A\cap \{0,1,\dots,n\}| 
\leq |B|^{\left\lfloor\log_N n\right\rfloor+1} 
< |B|\cdot  n^{1-\frac c{\log\log k}},$$
which is less than $n^{1-\frac{c'}{\log\log k}}$
for any $c'<c$ and sufficiently large $n$.
This proves the theorem after renaming $c'$ as $c$.
\end{proof}

\section{Relation to Arithmetic Kakeya Conjecture:
Proof of \texorpdfstring{\Cref{thm:arithmetic_kakeya}}{Theorem \ref{thm:arithmetic_kakeya}}}
\label{sec:kakeya}
In this section, we prove \Cref{thm:arithmetic_kakeya},
which we reproduce the statement below.

\kakeya*

\begin{proof}
Fix $\delta>0$ to be selected later.
We take $r_1,\dots,r_\ell \in \QQ\cup\{\infty\}$ 
satisfying \Cref{conj:arithmetic_kakeya}, 
that is, the inequality $|\pi_{-1}(B)| \leq \max_{i=1}^\ell 
|\pi_{r_i}(B)|^{1+\delta}$ holds for all finite subsets 
$B\subseteq \ZZ^2$.
Note that $r_i\neq -1$ for all $i$.

We first argue that we may assume $r_1,\dots,r_\ell>-1$
and $r_1,\dots,r_\ell \neq\infty$.
To do this, we note that for any M\"obius transformation 
$\phi(z) = \frac{az+b}{cz+d}$ such that $\phi(-1)=-1$,
applying linear transformation $\ttwomat dcba$
on $B$ allows us to replace $r_1,\dots,r_\ell$
with $\phi(r_1), \dots, \phi(r_\ell)$.
If elements in $B$ are now rational, we clear the denominator 
so that every element in $B$ is in $\mathbb Z^2$.
We now pick $\phi$ of the form $\phi(z) = \frac{z+1}{z-u}-1$
where $u>-1$ is a rational number such that $(-1,u)$ does not contain any 
of $r_1,\dots,r_\ell$.
This satisfies $\phi(r_i) > -1$
and $\phi(r_i)\neq\infty$ for all $1\leq i\leq\ell$.

Let $r_i = \frac{a_i}{b_i}$ for each $i$ where 
$a_i$ and $b_i$ are integers with $b_i>0$.
Let $L=\operatorname{lcm}(b_1,\dots,b_\ell)$ and
choose $k_0=L \max_{i=1}^\ell (r_i+1) + 2$.
These choices are so that $L(r_i+1)\in\{1,2,\dots,k-1\}$
for all $1\leq i\leq\ell$.
Let $k>k_0$ and $A$ be a $k$-AP covering set.
In particular, for any $x>n_0$, there exists $d(x)$
such that $x-d(x)$, $x-2d(x)$, $\dots$, $x-(k-1)d(x)$ 
are all in $A$.
We then define 
$$B = \{(x-L\,d(x), -L\,d(x)) : n_0<x\leq n\}.$$
Clearly $|\pi_{-1}(B)|\geq n-n_0$. 
The set $\pi_{r_i}(B)$ consists of elements of the form
$$x - L(r_i+1)\, d(x),$$
which is in $A$ because $L(r_i+1)\in \{1,2,\dots,k-1\}$
by our choice of $k$ and $L$.
Thus, we have that $|\pi_{r_i}(B)| \leq |A\cap\{0,1,\dots,n\}|$.
Therefore,
$$n-n_0 \leq |\pi_{-1}(B)| \leq 
\max_{i=1}^\ell |\pi_{r_i}(B)|^{1+\delta} \leq  
(|A\cap\{0,1,\dots,n\}|)^{1+\delta},$$
which implies that 
$$|A\cap\{0,1,\dots,n\}| \geq (n-n_0)^{1/(1+\delta)},$$
so picking any $\delta$ such that $\frac 1{1+\delta} > 1-\eps$ 
makes this quantity greater than $n^{1-\eps}$ 
for all sufficiently large $n$.
\end{proof}

\section*{Acknowledgment}
An earlier version of this paper was written as part of 
MIT Project Lab in Mathematics (18.821). 
The author thanks Pico Gilman, Matthew Halm, Noah Kravitz,
and Wei Zhang for helpful discussions and feedback on the paper.

All mathematical content in this paper is due to the author.
The author used ChatGPT and Claude to help identify typographical and mathematical errors.
\printbibliography
\end{document}